\documentclass{article}[12pt]

\usepackage{graphicx,tikz,color}
\usepackage{amsmath,amsfonts,amssymb,amsthm,latexsym}
\usepackage{enumerate}

\newtheorem{theorem}{Theorem}[section]

\newtheorem{corollary}[theorem]{Corollary}
\newtheorem{remark}[theorem]{Remark}
\newtheorem{lemma}[theorem]{Lemma}

\newcommand{\lp}{\left(}
\newcommand{\rp}{\right)}
\newcommand{\la}{\left\{}
\newcommand{\ra}{\right\}}
\newcommand{\oa}{o_{G[A]}}
\DeclareMathOperator{\core}{\rm Core}
\DeclareMathOperator{\es}{{\rm es}_\Delta}
\newcommand{\san}[1]{{\color{black} #1}}

\begin{document}

\title{On the $\Delta$-edge stability  number of graphs}

\author{Saieed Akbari$^{a}$\thanks{Email: \texttt{s\textunderscore akbari@sharif.edu}}\and Reza Hosseini Dolatabadi$^{b}$ \thanks{Email: \texttt{nimadolat80@gmail.com}}\and Mohsen Jamaali$^{a}$\thanks{Email: \texttt{mohsen\textunderscore djamali@yahoo.com}} \and Sandi Klav\v{z}ar $^{e,f,g}$ \thanks{Email: \texttt{sandi.klavzar@fmf.uni-lj.si}}   \and Nazanin Movarraei$^{c}$\thanks{Email: \texttt{nazanin.movarraei@gmail.com}}\smallskip }

\date{}

\maketitle

\begin{center}
$^a$ Department of Mathematical Sciences, Sharif University of Technology\\
\medskip

$^b$ Department of Computer Engineering, Sharif University of Technology\\
\medskip

$^c$ School of Mathematics, Institute for Research in Fundamental Sciences\\ 
\medskip

$^e$
Faculty of Mathematics and Physics, University of Ljubljana, Slovenia\\
\medskip
$^f$
Institute of Mathematics, Physics and Mechanics, Ljubljana, Slovenia\\
\medskip
$^g$
Faculty of Natural Sciences and Mathematics, University of Maribor, Slovenia\\
\end{center}

\medskip

\begin{abstract}
The $\Delta$-edge stability number ${\rm es}_{\Delta}(G)$ of a graph $G$ is the minimum number of edges of $G$ whose removal results in a subgraph $H$ with $\Delta(H) = \Delta(G)-1$. Sets whose removal results in a subgraph with smaller maximum degree are called mitigating sets. It is proved that there always exists a mitigating set which induces a disjoint union of paths of order $2$ or $3$. Minimum mitigating sets which induce matchings are characterized. It is proved that to obtain an upper bound of the form ${\rm es}_{\Delta}(G) \leq c |V(G)|$ for an arbitrary graph $G$ of given maximum degree $\Delta$, where $c$ is a given constant, it suffices to prove the bound for $\Delta$-regular graphs. Sharp upper bounds of this form are derived for regular graphs. It is proved that if $\Delta(G) \geq\frac{|V(G)|-2}{3}$ or the induced subgraph on maximum degree vertices has a $\Delta(G)$-edge coloring, then ${\rm es}_{\Delta}(G) \le \san{\lceil |V(G)|/2\rceil}$.
\end{abstract}


\noindent
AMS Subj.\ Class.\ (2020): 05C75, 05C07, 05C70, 05C15

\newpage
\section{Introduction}

Let $G$ be a \san{non-empty} graph and let $\Delta(G)$ denote its maximum degree. The {\em {$\Delta$-edge stability number}}, $\es(G)$, of $G$, is the minimum number of edges of $G$ \san{whose removal results} in a subgraph $H$ with $\Delta(H) = \Delta(G)-1$. This graph invariant has been for the first time investigated by Borg and Fenech in~\cite{borg-2019}. The vertex version of the problem, that is, the $\Delta$-vertex stability number, has been studied in~\cite{borg-2022, borg-2017}. Furthermore, over the last few years, the corresponding problems for the chromatic number and the chromatic index were respectively investigated in~\cite{akbari-2020, bresar-2020} and~\cite{akbari-2022, akbari-2023, kemnitz-2018, knor-2022, lei-2023}. 

The above discussion naturally \san{falls under} the following broader framework recently proposed in~\cite{kemnitz-2022} and further elaborated in~\cite{alikhani-2023, kemnitz-2024}. For an arbitrary graph invariant $\tau$, the $\tau$-vertex stability number (the $\tau$-edge stability number) is the minimum number of vertices (edges) whose removal results in a subgraph $H$ with $\tau(H) \ne \tau(G)$. The corresponding minimum number of vertices and edges are respectively denoted by ${\rm vs}_{\tau}(G)$ and ${\rm es}_{\tau}(G)$. Following this general \san{framework, in this paper we use} the notation $\es(G)$, although we should add that the $\Delta$-edge stability number and the $\Delta$-vertex stability number of a graph $G$ were \san{also} denoted by $\lambda_{\rm e}(G)$ and $\lambda(G)$. 

In the seminal paper~\san{\cite{borg-2019},} the focus was on the upper bounds of the $\Delta$-edge stability number in terms of the size of the graph, the maximum degree, and the number of vertices of maximum degree. In this paper, we continue with the exploration of the $\Delta$-edge stability number. Our main goal is to find tight bounds for $\es(G)$ based on the order of the graph. 

The paper is organized as follows. In the rest of the introduction, we briefly define notations used in the paper and recall a result to be used later on. In the subsequent section, we prove several general properties for sets of edges whose removal decreases the maximum degree. In particular, we prove that there always exists such a set which induces a disjoint union of paths of order $2$ or $3$, and characterize smallest such sets which are matchings. In Section~\san{\ref{sec:regularization},} we prove that to obtain an upper bound of the form $\es(G) \leq c |V(G)|$ for an arbitrary graph $G$ of given maximum degree $\Delta$, where $c$ is a given constant, it \san{suffices} to prove the bound for $\Delta$-regular graphs. \san{This result is an analogue of~\cite[Theorem~3]{borg-2022} the $\Delta$-vertex stability version.} In Section~\san{\ref{sec:regular},} sharp upper \san{bounds} for regular graphs are derived. In the final section, we prove that the $\Delta$-edge stability number is bounded from the above by one-half of the order for each graph in which vertices of maximum degree induce a Class 1 graph, as well as for graphs $G$ with $\Delta(G) \geq\frac{|V(G)|-2}{3}$. 

Throughout this \san{paper,} all graphs are finite and simple, that is, with no loops and multiple edges, and moreover, with at least one edge. Let $G=(V(G), E(G))$ be a graph. The degree of a vertex $u$ of $G$ is denoted by $d_G(u)$. Further, $\delta(G)$ is the minimum degree of $G$. The subgraph of $G$ induced by a set $A$ of vertices and/or edges will be denoted by $G \left[ A \right]$. The number of edges between two disjoint sets of vertices $S,T\subseteq V(G)$ is denoted by $e(S,T)$. The open neighborhood $N_G(v)$ of a vertex $v\in V(G)$ is the set of neighbors of $v$, the closed neighborhood of $v$ is $N_G[v] = N_G(v)\cup \{v\}$. If $S\subseteq V(G)$, then the open and the closed neighborhood of $S$ are the respective sets $N_G(S) = \cup_{v\in S} N_G(v)$ and $N_G[S] = \cup_{v\in S} N_G[v]$. If the graph $G$ is clear from the context, we may omit the subscript $G$ in the above notation. The {\em independence number} of $G$ is denoted by $\alpha(G)$ and its matching number by $\alpha'(G)$. The {\em odd girth} of $G$ is the length of a shortest odd cycle in $G$ and is denoted by ${\rm og}(G)$. 

The {\em core} of $G$ is the set of vertices of $G$ of maximum degree and is denoted by {\em $\core(G)$}. Clearly, if $G$ is regular, then $\core(G) = V(G)$. If $S\subseteq E(G)$ is such that \san{$\Delta(G-S) \le \Delta(G) - 1$}, then we say that $S$ is a {\em mitigating set} of $G$. In that case, we also say that $G[S]$ is a {\em mitigating subgraph} of $G$. 
 
Given a graph $G$, a function $c:E(G) \to \{c_1,\ldots , c_k\}$ with $c(e) \neq c(f)$ for any two adjacent edges $e$ and $f$ is a  {\emph{proper $k$-edge coloring}} of $G$. The minimum $k$ for which $G$ admits a proper $k$-edge coloring is the {\emph{chromatic index}} of $G$, and denoted by $\chi'(G)$.
We let $[k]=\{1,\ldots,k\}$.
For any $i \in [\chi'(G)]$, let $C_i$ denote the set of all edges of $G$ that are colored by $c_i$ in the proper edge coloring $c$. 
For any $v \in V(G)$, let $c(v)$ denote the set of colors appearing in $v$.
Vizing's Theorem~\cite{vizing-1964} states that the chromatic index of an arbitrary simple graph \san{is  $\Delta(G)$ or $\Delta(G)+1$}. Graphs with $\chi'(G)=\Delta(G)$ are said to be {\em Class $1$}, while graphs with $\chi'(G)=\Delta(G)+1$ are said to be {\em Class $2$}.

Throughout the \san{following,} we will use Tutte's Theorem\san{~\cite{tutte-1947}}, which states that a graph has a perfect matching if and only if $o(G-S) \le |S|$ for every $S\subseteq V(G)$, where $o(H)$ denotes the number of odd components of a graph $H$. We conclude the preliminaries by recalling the following result to be used later on. 

\begin{theorem} {\rm \cite[Theorem 2.8]{borg-2019}} 
\label{thm:exact es-Delta}
If $G$ is a graph, then $\es(G) = |\core(G)| - \alpha'(G[\core(G)])$.
\end{theorem}

\section{Properties of  mitigating sets}
\label{sec:Hall}

In this section, we first prove that one can always find a mitigating subgraph of $G$ whose each component is $P_2$ or $P_3$, and give an upper bound on the $\Delta$-edge stability number involving the independence number. To derive both results, Hall's Theorem will be applied. Afterwards, we characterize in two ways minimum mitigating sets which are matchings.

\begin{theorem}
\label{thm:components}
If $G$ is a graph of order $n$, then the following properties hold. 
\begin{enumerate}
\item[(i)] There exists a mitigating subgraph of $G$ whose each component is $P_2$ or $P_3$.
\item[(ii)] $\es(G)\leq n-\alpha(G)$, and the bound is sharp.
\end{enumerate}
\end{theorem}

\begin{proof}
(i) Let $M$ be a maximum matching of $G[\core(G)]$, and let $A \subseteq \core(G)$ be the set of vertices of $\core(G)$ which are not saturated by $M$. As $M$ is a maximum matching of $G[\core(G)]$, the set $A$ is independent. Furthermore, since the vertices in $A$ are of maximum degree, $|N(A')| \ge |A'|$ holds for each $A'\subseteq A$. Therefore, by Hall's Theorem, there exists a matching $M'$ in $G$ that saturates $A$. By applying Theorem~\ref{thm:exact es-Delta} we have
\begin{align*}
\es(G) & \le  |M\cup M'| \le |M| + |M'| \\
& = \alpha'(G[\core(G)]) + (|\core(G)| - \san{2\alpha'(G[\core(G)]})) \\
& = |\core(G)| - \alpha'(G[\core(G)]) \\
& = \es(G)\,,
\end{align*}
hence equality must hold in the first line. Therefore, $M \cup M'$ is a mitigating set of $G$. 

Since $M$ is a maximum matching of $G[\core(G)]$, there is no pair of edges of $M'$ which meet an edge of $M$. Thus, any component of $M \cup M'$ is $P_2$ or $P_3$.

(ii) Let $I \subseteq V \lp G \rp$ be an independent set of size $\alpha(G)$, and $A = \core(G) \cap I$. By Hall's Theorem, there exists a matching $M$ in $G$ that saturates $A$.
    Assume $B \subseteq V(G) - I$ is the set of vertices \san{in} $V(G) - A$ which are saturated by $M$. Let $S$ be the set of edges containing $M$ as well as one edge adjacent to each vertex \san{in} $V(G) - (I \cup B)$. Clearly, $G - S$ has no vertex of degree $\Delta$. Since 
$$|S| = |M| + (n - |I \cup B|) = |M| + n - (\alpha(G) + |M|) = n - \alpha(G)\,,$$ 
the bound is proved. To see that it is sharp, consider an arbitrary regular, bipartite graph $G$ with a perfect matching. Then $\es(G) = \alpha(G) = n/2$. 
\end{proof}

To characterize minimum mitigating sets, we need a lemma which can be deduced from Tutte's Theorem.

\begin{lemma} \label{lem: tutte lemma}
    Let $G$ be a graph and $A \subseteq V(G)$. If for all $S \subseteq V(G)$ we have
    \[
        \oa(G - S) \leq |S|,
    \]
    then there exists a matching in $G$ that saturates $A$, where $\oa(G - S)$ is the number of odd components of $G - S$ which are contained \san{in $G[A]$}. 
\end{lemma}
\begin{proof}
    Let $n = |V(G)|$ and $B = V(G) - A$. Let $H$ be the graph obtained by the disjoint union of $K_n$ and $G$ and joining each vertex of $K_n$ to \san{each} vertex of $B$. Clearly, $H$ is of order $2n$. Moreover, as the order of $H$ is even, we see that $H$ has a perfect matching if and only if $G$ has a matching that saturates $A$. Note that maybe some edges in the matching have one endpoint in $A$ and another in $B$. Assume for the contrary that no matching in $G$ covers $A$. Therefore, by Tutte's \san{Theorem,} there exists $S \subseteq V(H)$ such that \san{$o(H-S) > |S|$. Since $|V(H)|$ is even, $|S|$ and $|o(H-S)|$ have the same parity. Therefore, the following inequality holds}
    \begin{equation} \label{eq: lemma ineq}
                o(H - S) \geq |S| + 2\,.
    \end{equation}
If $V(K_n) \subseteq S$, then $|S| \geq \frac{|V(H)|}{2}$ and thus, (\ref{eq: lemma ineq}) does not hold. Hence there exists $v \in V(K_n) - S$. Obviously, there are at least $|S| + 1$ odd components of $H - S$ that are contained \san{in $G[A]$}, which is a contradiction, and the lemma is proved.
\end{proof}

We now characterize minimum mitigating sets which are matchings as follows.  

\begin{theorem}
\label{thm:minimum-mitigating}
If $G$ is a graph, then the following statements are equivalent.
    \begin{enumerate}
        \item[(1)] $G$ has a matching that saturates $\core(G)$. 
        \item[(2)] $G$ has a minimum mitigating set which is a matching.         \item[(3)] For every $S \subseteq \core(G)$, $\es(G[N[S]]) \leq \frac{|V(G[N[S]])|}{2}$.
    \end{enumerate}
\end{theorem}

\begin{proof}
We first show that (1) and (2) are equivalent. It is clear that if there is a minimum mitigating set which is a matching, this matching saturates $\core(G)$. For the \san{converse}, suppose $M$ is a matching that saturates $\core(G)$ and let $L$ be a minimum mitigating set for $G$. We claim that $L$ could be chosen as a matching using induction on $\es(G)$. If $\es(G) = 1$, the claim is obvious. Hence, assume that $\es(G) \geq 2$. Clearly, $|M| \geq |L|$. If $L - M = \varnothing$, $L$ must be a matching and the claim is proved. Thus, assume $e \in L - M$. Since $|L| \geq 2$, we have $\Delta(G - e) = \Delta(G)$. Therefore, $M$ is a matching that saturates $\core(G- e)$. By the induction hypothesis, there exists a minimum mitigating set for $G -e$, say $L'$,  which is a matching and $|L'| = \es(G) - 1$. Obviously, $L'' = L' \cup \{e\}$ is a minimum mitigating set for $G$. If $e$ is not adjacent to any edge in $L'$, the claim is proved. Since $L'$ is \san{a matching} and $L''$ is a minimum mitigating set, $e$ is adjacent to at most one edge in $L'$, say $e'$. Let $e=uv$ and $e'= vw$ for some $u,v,w \in V(G)$. It is easy to see that $d(u) = \Delta$. Since $uv \notin M$, there exists $e_1 = uy_1 \in M$. Replace $e$ with $e_1$ in $L''$ and the resulting mitigating set is still minimum and is a matching, unless there exists  $e'_1 = y_1x_1 \in L''$ and $d(x_1)=\Delta$. Otherwise, if we remove $e'_1$ from $L''$, the resulting set of edges is still a mitigating set. Replace $e'_1$ with some $e_2 = x_1y_2 \in M$ to obtain a new minimum mitigating set and continue this operation and in each step replace $e'_i$ with $e_{i+1}$. Note that if $d(x_{i}) < \Delta$ or there is no edge in $L'$ saturating $y_{i+1}$, there is no need to continue the operation. Moreover, since $M$ is a matching, all $x_i$ and $y_i$ are distinct. Therefore, this operation will be stopped after finitely many steps. So, when the replacement is done, the resulting minimum mitigating set is a matching and the claim is proved.

We next prove that (1) and (3) are equivalent. \san{For any} $S \subseteq V(G)$, set $H(S) = G[N[S]]$.  It is clear that if $G$ has a matching that saturates $\core(G)$, then (3) holds. For the converse, set $n = |V(G)|$. Let \san{$A = \core(G)$.} By contradiction assume that $G$ has no matching that covers $A$. Therefore, by Lemma \ref{lem: tutte lemma}, there exists $S \subseteq V(G)$ such that
    \[
        \oa(G - S) \geq |S| + \san{1\,.}
    \]

 \san{Let} $C_1, C_2, \ldots , C_{|S| + 1} \subseteq A$ \san{be distinct} odd components contained \san{in $G[A]$}. Let $C = \bigcup_{i=1} ^ {|S| + 1} \san{V(C_i)}$. Clearly, $|V(H(C))| \leq |S| + \san{|C|}$. Since $V(C_i)$ has odd cardinality and consists of vertices of degree $\Delta$, we have
\[
    \es \lp \san{H(V(C_i))} \rp \geq \frac{|V(C_i)|+1}{2}\,.
\]
Therefore,
\[
    \es \lp H(C) \rp \geq \sum_{i=1}^{|S|+1} \frac{|V(C_i)|+1}{2} \geq \frac{\san{|C|}}{2} + \frac{|S| + 1}{2} > \frac{\san{|C|} + |S|}{2} \geq \frac{|V(H(C))|}{2}\,,
\]
a contradiction and we are done.
\end{proof}

\section{Graph's regularization and its applications}
\label{sec:regularization}

In this \san{section,} we prove that to derive an upper bound of the form $\es(G) \leq c |V(G)|$ for an arbitrary graph, where $c$ is a given constant, it suffices to prove the bound for regular graphs. \san{This is modelled on the work~\cite[Section~4]{borg-2022};  Lemma~\ref{lem: regularization lemma} is an analogue of Inequality (1) in~\cite{borg-2022}, and Theorem~\ref{thm:reduction-to-regular} is an analogue of~\cite[Theorem~3]{borg-2022}. The idea is the same.} 

For a graph $G$ we construct its {\em regularization} $R(G)$ as follows. If $G$ is regular, then set $R(G) = G$. Assume now that $G$ is not regular and set $A_G = V(G) - \core(G)$. Then the graph $G^{(1)}$ is obtained from two disjoint copies of $G$, say $G'$ and $G''$, by adding a matching between the corresponding vertices in $A_{G'}$ and $A_{G''}$. Note that $\Delta(G^{(1)}) = \Delta(G)$ and $\delta(G^{(1)}) = \delta(G) + 1$. If $G^{(1)}$ is not yet regular, we repeat the same construction on $G^{(1)}$ to \san{obtain} $G^{(2)}$. Repeating the construction $\Delta(G) - \delta(G)$ \san{times,} we arrive at the regularization $R(G)$ of $G$: 
$$R(G) = G^{(\Delta(G) - \delta(G))}\,,$$ 
which is a $\Delta(G)$-regular graph. 

A key property of $R(G)$ is the following. 

\begin{lemma} \label{lem: regularization lemma}
If $G$ is a graph, then  
\[
    \frac{|V(R(G))|}{\es(R(G))} \leq \frac{|V(G)|}{\es(G)}\,.
\]
\end{lemma}

\begin{proof}
There is nothing to prove if $G$ is regular, hence assume in the rest that $\Delta(G)-\delta(G)\ge 1$. We first claim that
    \[
        \frac{|V(G^{(1)})|}{\es(G^{(1)})} \leq \frac{|V(G)|}{\es(G)}\,.
    \]
Let $M$ be a mitigating set of $G^{(1)}$ and let $G'$ and $G''$ be the two copies of $G$ in $G^{(1)}$. As every edge between $G'$ and $G''$ connects vertices of \san{non-maximum} degree in $G'$ and $G''$, the sets $M \cap E(G')$ and $M \cap E(G'')$ are mitigating sets of $G'$ and $G''$ respectively. Thus, $\es(G^{(1)}) \geq 2\es(G)$. Since $|V(G^{(1)})| = 2|V(G)|$, the claim is proved. Proceeding by induction we analogously infer that 
    \[
        \frac{|V(G^{(i)})|}{\es(G^{(i)})} \leq \frac{|V(G)|}{\es(G)}
    \]
holds for each $i \in \{2, \ldots, \Delta(G) - \delta(G)\}$. Thus the assertion. 
\end{proof}

The announced reduction to regular graphs now reads as follows. 

\begin{theorem}
\label{thm:reduction-to-regular}
If there exists a constant $0 < c_{\Delta} < 1$, such that $\es(H) \leq  c_{\Delta}|V(H)|$ holds for every $\Delta$-regular graph $H$, then $\es(G) \leq c_{\Delta}|V(G)|$ holds for every graph $G$ with $\Delta(G) = \Delta$.
\end{theorem}

\begin{proof}
Let $G$ be an arbitrary graph with $\Delta(G) = \Delta$. As there is nothing to prove if $G$ is regular, assume this is not the case. Then by Lemma~\ref{lem: regularization lemma} and the theorem's assumption we get
\begin{align*}
\es(G) & \le \es(R(G)) \frac{|V(G)|}{|V(R(G))|} \le c_{\Delta} |V(R(G))| \frac{|V(G)|}{|V(R(G))|} = c_{\Delta} |V(G)|
\end{align*}
and we are done. 
\end{proof}

Another applications of Lemma~\ref{lem: regularization lemma} is the following bound on the $\Delta$-edge stability number of a graph in terms of its odd girth. 

\begin{theorem}
If $G$ \san{is} a graph of order $n$ and ${\rm og}(G) = 2k+1$, $k\ge 1$, then $\es(G)\leq\frac{k+1}{2k+1}n$.
\end{theorem}

\begin{proof}
We claim that ${\rm og}(G^{(i)})\ge 2k+1$ holds for each $i \in [\Delta - \delta]$. Consider first $G^{(1)}$ and let $W_1$ be an arbitrary odd closed walk in it. Assume that $W_1$ passes through $G'$ as well as through $G''$. Let $e=u'u''$ and $f=v'v''$ be two edges of $W_1$ between $G'$ and $G''$ (where $u', v'\in V(G')$ and $u'',v''\in V(G'')$) occurring consecutively on $W_1$. Let $W''_{u''v''}$ be the \san{$u''$-$v''$-subwalk} of $W_1$. Then $W''_{u''v''}$ is contained in $G''$. Now replace the $u'-u''-W''_{u''v''}-v''-v'$ subwalk of $W_1$ by the walk $u'-W'_{u'v'}-v'$, where $W'_{u'v'}$  is the isomorphic copy of 
$W''_{u''v''}$ in $G'$. Repeating this process if necessary, we arrive at a closed walk of $G^{(1)}$ which lies completely in $G'$ and is shorter (or of equal length if $W_1$ already lies completely in $G'$) than $W_1$. As $G'$ is isomorphic to $G$, this \san{proves} the claim for $G^{(1)}$. The argument for $G^{(i)}$, $i\in \{2, \ldots, \Delta - \delta\}$ is then analogous. 


$R(G)$ is a regular graph. \san{From the introduction of ~\cite{kano-1986} we recall that Hajnal~\cite{hajnal-1965} and Tutte~\cite{tutte-1952} proved that a regular graph has a $\{1,2\}$-factor $F$ (each of whose components are regular).} Clearly for any component $C$ of $F$, one can remove $\lceil{\frac{|V(C)|}{2}}\rceil$ edges of $C$ which saturate $V(C)$. Since $\lceil{\frac{|V(C)|}{2}}\rceil \leq \frac{k+1}{2k + 1}|V(C)|$, we obtain
    \[
        \es(R(G)) \leq \frac{k+1}{2k + 1}|V(R(G))|\,.
    \]
By Lemma~\ref{lem: regularization lemma} we then get 
\begin{align*}
\es(G) & \le \es(R(G)) \frac{|V(G)|}{|V(R(G))|} \le \frac{k+1}{2k + 1} |V(R(G))| \frac{|V(G)|}{|V(R(G))|} = \frac{k+1}{2k + 1} |V(G)|
\end{align*}
and we are done.     
\end{proof}

\section{The regular case}
\label{sec:regular}

In view of Theorem~\ref{thm:reduction-to-regular}, in this section we take a closer look \san{at} regular graphs. For this \san{sake,} we first recall the following fundamental result due to Henning and Yeo. 

\begin{theorem}[Henning and Yeo, \cite{henning-2007}] 
\label{thm:mike&anders}
Let $G$ be a connected, $k$-regular graph of order \san{$n$, where $k \geq 2$. If $k$} is even, then 
    \[
        \alpha'(G) \geq \min \left\{ \frac{(k^2 + 4)n}{2(k^2 + k + 2)}, \frac{n-1}{2} \right\}\,,
    \]
and if \san{$k$ is odd,} then 
    \[
        \alpha'(G) \geq \frac{(k^3 - k^2 - 2)n -2k + 2}{2(k^3 - 3k)}\,.
    \]
Moreover, both bounds are tight.
\end{theorem}

From Theorem~\ref{thm:mike&anders} we can deduce the following consequence.

\begin{corollary} \label{cor: es for regulars}
Let $G$ be a connected, $k$-regular graph of order \san{$n$, where $k \geq 2$. If $k$} is even, then 
    \[
        \es(G) \leq \max \left\{ \left(1-\frac{k^2 + 4}{2(k^2 + k + 2)}\right) n, \frac{n+1}{2} \right\}\,,
    \]
and if $k\ge 3$ is odd, then 
    \[
        \es(G) \leq \frac{(k^3 + k^2  -6k + 2)n +2k - 2}{2(k^3 - 3k)}\,.
    \]
Moreover, both bounds are tight.
\end{corollary}

\begin{proof}
By Theorem~\ref{thm:exact es-Delta} and the assumption that $G$ is regular, we have  $\es(G) = n - \alpha'(G)$. Therefore, if $k$ is even, then using Theorem~\ref{thm:mike&anders} we can estimate as follows: 
\begin{align*}
\es(G) & \le n - \min \left\{ \frac{(k^2 + 4)n}{2(k^2 + k + 2)}, \frac{n-1}{2} \right\} \\
& = \max \left\{ n - \frac{(k^2 + 4)n}{2(k^2 + k + 2)}, n - \frac{n-1}{2} \right\} \\
& = \max \left\{ \left(1-\frac{k^2 + 4}{2(k^2 + k + 2)}\right) n, \frac{n+1}{2} \right\}\,.
\end{align*}
The estimate for odd $k$ is derived analogously. 

The tightness follows by the tightness of the estimates from Theorem~\ref{thm:mike&anders}.
\end{proof}

As discussed in~\cite{henning-2007}, the bound $(n-1)/2$ from Theorem~\ref{thm:mike&anders} is only necessary to cover some cases when $n$ is very small or $k=2$. In particular, it is not needed for $k\ge4$, cf.~\cite[Corollary 1]{henning-2007}. Therefore, we can also state the following easier-to-read corollary for all ``non-trivial" even $k$. 

\begin{corollary} \label{cor: es for even regulars at least 4}
If $G$ is a connected graph of order $n$ \san{and of even} maximum degree \san{$k \geq 4$}, then 
    \[
        \es(G) \leq \left(1-\frac{k^2 + 4}{2(k^2 + k + 2)}\right) n\,.
    \]
\end{corollary}

If $k=4$, then the bound of Corollary~\ref{cor: es for even regulars at least 4} reads as  $\es(G) \leq \frac{6}{11} n$. We next construct  $2$-connected, $4$-regular graphs \san{which attain} this bound. Let $H_i = K_5 - e$, $i\in [2k]$, and let $G_k$ be the graph of order $n = 11k$ as shown in Fig.~\ref{fig:Gk}. 

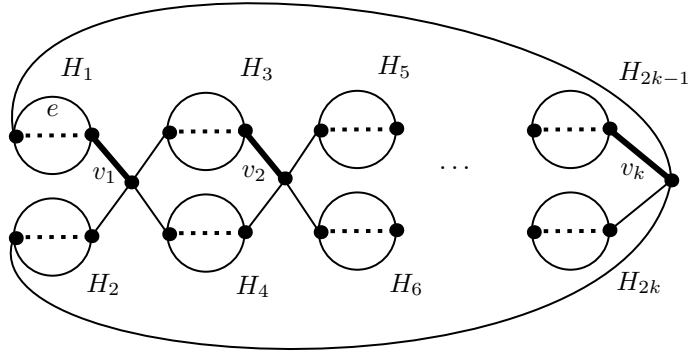
\begin{figure}[ht!]
\centering
\tikzset{every picture/.style={line width=0.75pt}} 
\begin{tikzpicture}[x=0.75pt,y=0.75pt,yscale=-1,xscale=1]

\draw   (57.03,70.72) .. controls (56.98,60.02) and (65.62,51.3) .. (76.33,51.25) .. controls (87.03,51.2) and (95.75,59.84) .. (95.8,70.55) .. controls (95.85,81.26) and (87.21,89.98) .. (76.5,90.03) .. controls (65.79,90.07) and (57.07,81.43) .. (57.03,70.72) -- cycle ;
\draw [line width=1.5]  [dash pattern={on 1.69pt off 2.76pt}]  (57.03,70.72) -- (95.8,70.55) ;
\draw  [line width=5.25] [line join = round][line cap = round] (58,71.2) .. controls (58,71.2) and (58,71.2) .. (58,71.2) ;
\draw  [line width=5.25] [line join = round][line cap = round] (96,70.2) .. controls (96,70.2) and (96,70.2) .. (96,70.2) ;
\draw   (57.03,121.72) .. controls (56.98,111.02) and (65.62,102.3) .. (76.33,102.25) .. controls (87.03,102.2) and (95.75,110.84) .. (95.8,121.55) .. controls (95.85,132.26) and (87.21,140.98) .. (76.5,141.03) .. controls (65.79,141.07) and (57.07,132.43) .. (57.03,121.72) -- cycle ;
\draw [line width=1.5]  [dash pattern={on 1.69pt off 2.76pt}]  (57.03,121.72) -- (95.8,121.55) ;
\draw  [line width=5.25] [line join = round][line cap = round] (58,122.2) .. controls (58,122.2) and (58,122.2) .. (58,122.2) ;
\draw  [line width=5.25] [line join = round][line cap = round] (96,121.2) .. controls (96,121.2) and (96,121.2) .. (96,121.2) ;
\draw   (134.03,68.72) .. controls (133.98,58.02) and (142.62,49.3) .. (153.33,49.25) .. controls (164.03,49.2) and (172.75,57.84) .. (172.8,68.55) .. controls (172.85,79.26) and (164.21,87.98) .. (153.5,88.03) .. controls (142.79,88.07) and (134.07,79.43) .. (134.03,68.72) -- cycle ;
\draw [line width=1.5]  [dash pattern={on 1.69pt off 2.76pt}]  (134.03,68.72) -- (172.8,68.55) ;
\draw  [line width=5.25] [line join = round][line cap = round] (135,69.2) .. controls (135,69.2) and (135,69.2) .. (135,69.2) ;
\draw  [line width=5.25] [line join = round][line cap = round] (173,68.2) .. controls (173,68.2) and (173,68.2) .. (173,68.2) ;
\draw   (134.03,119.72) .. controls (133.98,109.02) and (142.62,100.3) .. (153.33,100.25) .. controls (164.03,100.2) and (172.75,108.84) .. (172.8,119.55) .. controls (172.85,130.26) and (164.21,138.98) .. (153.5,139.03) .. controls (142.79,139.07) and (134.07,130.43) .. (134.03,119.72) -- cycle ;
\draw [line width=1.5]  [dash pattern={on 1.69pt off 2.76pt}]  (134.03,119.72) -- (172.8,119.55) ;
\draw  [line width=5.25] [line join = round][line cap = round] (135,120.2) .. controls (135,120.2) and (135,120.2) .. (135,120.2) ;
\draw  [line width=5.25] [line join = round][line cap = round] (173,119.2) .. controls (173,119.2) and (173,119.2) .. (173,119.2) ;
\draw   (210.03,67.72) .. controls (209.98,57.02) and (218.62,48.3) .. (229.33,48.25) .. controls (240.03,48.2) and (248.75,56.84) .. (248.8,67.55) .. controls (248.85,78.26) and (240.21,86.98) .. (229.5,87.03) .. controls (218.79,87.07) and (210.07,78.43) .. (210.03,67.72) -- cycle ;
\draw [line width=1.5]  [dash pattern={on 1.69pt off 2.76pt}]  (210.03,67.72) -- (248.8,67.55) ;
\draw  [line width=5.25] [line join = round][line cap = round] (211,68.2) .. controls (211,68.2) and (211,68.2) .. (211,68.2) ;
\draw  [line width=5.25] [line join = round][line cap = round] (249,67.2) .. controls (249,67.2) and (249,67.2) .. (249,67.2) ;
\draw   (210.03,118.72) .. controls (209.98,108.02) and (218.62,99.3) .. (229.33,99.25) .. controls (240.03,99.2) and (248.75,107.84) .. (248.8,118.55) .. controls (248.85,129.26) and (240.21,137.98) .. (229.5,138.03) .. controls (218.79,138.07) and (210.07,129.43) .. (210.03,118.72) -- cycle ;
\draw [line width=1.5]  [dash pattern={on 1.69pt off 2.76pt}]  (210.03,118.72) -- (248.8,118.55) ;
\draw  [line width=5.25] [line join = round][line cap = round] (211,119.2) .. controls (211,119.2) and (211,119.2) .. (211,119.2) ;
\draw  [line width=5.25] [line join = round][line cap = round] (249,118.2) .. controls (249,118.2) and (249,118.2) .. (249,118.2) ;
\draw   (317.03,67.72) .. controls (316.98,57.02) and (325.62,48.3) .. (336.33,48.25) .. controls (347.03,48.2) and (355.75,56.84) .. (355.8,67.55) .. controls (355.85,78.26) and (347.21,86.98) .. (336.5,87.03) .. controls (325.79,87.07) and (317.07,78.43) .. (317.03,67.72) -- cycle ;
\draw [line width=1.5]  [dash pattern={on 1.69pt off 2.76pt}]  (317.03,67.72) -- (355.8,67.55) ;
\draw  [line width=5.25] [line join = round][line cap = round] (318,68.2) .. controls (318,68.2) and (318,68.2) .. (318,68.2) ;
\draw  [line width=5.25] [line join = round][line cap = round] (356,67.2) .. controls (356,67.2) and (356,67.2) .. (356,67.2) ;
\draw   (317.03,118.72) .. controls (316.98,108.02) and (325.62,99.3) .. (336.33,99.25) .. controls (347.03,99.2) and (355.75,107.84) .. (355.8,118.55) .. controls (355.85,129.26) and (347.21,137.98) .. (336.5,138.03) .. controls (325.79,138.07) and (317.07,129.43) .. (317.03,118.72) -- cycle ;
\draw [line width=1.5]  [dash pattern={on 1.69pt off 2.76pt}]  (317.03,118.72) -- (355.8,118.55) ;
\draw  [line width=5.25] [line join = round][line cap = round] (318,119.2) .. controls (318,119.2) and (318,119.2) .. (318,119.2) ;
\draw  [line width=5.25] [line join = round][line cap = round] (356,118.2) .. controls (356,118.2) and (356,118.2) .. (356,118.2) ;
\draw  [line width=5.25] [line join = round][line cap = round] (116,94.2) .. controls (116,94.2) and (116,94.2) .. (116,94.2) ;
\draw  [line width=5.25] [line join = round][line cap = round] (193,92.2) .. controls (193,92.2) and (193,92.2) .. (193,92.2) ;
\draw  [line width=5.25] [line join = round][line cap = round] (387,93.2) .. controls (387,93.2) and (387,93.2) .. (387,93.2) ;
\draw [line width=2.25]    (95.8,70.55) -- (116,94.2) ;
\draw    (95.8,121.55) -- (116,94.2) ;
\draw    (116,94.2) -- (134.03,68.72) ;
\draw [line width=2.25]    (172.8,68.55) -- (193,93.2) ;
\draw    (172.8,119.55) -- (193,93.2) ;
\draw    (116,94.2) -- (134.03,119.72) ;
\draw    (193,93.2) -- (210.03,118.72) ;
\draw    (210.03,67.72) -- (193,93.2) ;
\draw [line width=2.25]    (355.8,67.55) -- (365.44,75.48) -- (387,93.2) ;
\draw    (355.8,118.55) -- (387,93.2) ;
\draw    (57.03,70.72) .. controls (36,-38.8) and (384,0.2) .. (387,93.2) ;
\draw    (57.03,121.72) .. controls (30,192.2) and (365,210.2) .. (387,93.2) ;

\draw (269,84.4) node [anchor=north west][inner sep=0.75pt]    {$\dotsc $};
\draw (79,30.4) node [anchor=north west][inner sep=0.75pt]    {$H_{1}$};
\draw (92,138.4) node [anchor=north west][inner sep=0.75pt]    {$H_{2}$};
\draw (167,139.4) node [anchor=north west][inner sep=0.75pt]    {$H_{4}$};
\draw (169,30.4) node [anchor=north west][inner sep=0.75pt]    {$H_{3}$};
\draw (238,29.4) node [anchor=north west][inner sep=0.75pt]    {$H_{5}$};
\draw (244,138.4) node [anchor=north west][inner sep=0.75pt]    {$H_{6}$};
\draw (359,32.4) node [anchor=north west][inner sep=0.75pt]    {$H_{2k-1}$};
\draw (358,137.4) node [anchor=north west][inner sep=0.75pt]    {$H_{2k}$};
\draw (95,85.4) node [anchor=north west][inner sep=0.75pt]    {$v_{1}$};
\draw (170,84.4) node [anchor=north west][inner sep=0.75pt]    {$v_{2}$};
\draw (360,83.4) node [anchor=north west][inner sep=0.75pt]    {$v_{k}$};
\draw (72,53.4) node [anchor=north west][inner sep=0.75pt]    {$e$};
\end{tikzpicture}
    \caption{The graph $G_k$}
    \label{fig:Gk}
\end{figure}
    
We claim that $\alpha'(G) = 5k$ and $\es(G) = 6k$. Since $\alpha'(K_5 - e) = 2$, and every matching of $G$ has at most one edge incident with $v_j$, $j\in [k]$, we infer that $\alpha'(G) \leq 2(2k) + k = 5k$. Note that a $2$-matching of each $H_i$ together with the bold edges from Fig.~\ref{fig:Gk} represent a maximum matching of size $5k$ in $G$. So, $\alpha'(G) = 5k$. Since $G$ is 4-regular, by Theorem~\ref{thm:exact es-Delta}, $\es(G) = 6k$ and hence the conclusion is reached after a direct computation.

From \san{Theorem~\ref{thm:reduction-to-regular} and} the second estimate of Corollary~\ref{cor: es for regulars}, we can deduce also the following consequence. 

\begin{corollary} \label{cor: es for odd delta}
    For each $\epsilon > 0$ and sufficiently large $n$, for every graph $G$ of order $n$ with odd maximum degree $k$, we have
    \[
        \es(G) \leq \lp \frac{k^3 + k^2  -6k + 2}{2(k^3 - 3k)} + \epsilon \rp n.
    \]
\end{corollary}

We know the following  function is monotone
\[
    f(k) = \frac{k^3 + k^2  -6k + 2}{2(k^3 - 3k)}\,.
\]
Moreover,
\[
    \lim_{k \rightarrow \infty} f(k) = \frac{1}{2}\,,
\]
By Corollary~\ref{cor: es for even regulars at least 4} \san{for any even $k > 2$} we have
\[
    \es(G) \leq \frac{n+1}{2}.
\]
Combining Corollaries~\ref{cor: es for even regulars at least 4} and~\ref{cor: es for odd delta} with the above discussion, and having Theorem~\ref{thm:reduction-to-regular} in mind, we get the following result by setting $k=3$.

\begin{corollary}
    For each $\epsilon > 0$ and sufficiently large $n$, for every connected graph $G$ of order $n$, we have
    \[
        \es(G) \leq \lp \frac{5}{9} + \epsilon \rp n.
    \]
\end{corollary}

\section{Graphs \san{whose} $\Delta$-edge stability number \san{is} at most half the order}
\label{sec:1/2-bound}

In this \san{section,} we are interested in families of graphs for which the $\Delta$-edge stability number is bounded from the above by one half of the order. First, as a consequence of Theorem~\ref{thm:components}(ii),  this holds true for bipartite graphs. 

\begin{corollary} 
If $G$ is a bipartite graph of order $n$, then $\es(G)\leq n/2$.
\end{corollary} 

Our next goal is to show that this bound also holds for each graph $G$ \san{such that} $G[\core(G)]$ is Class 1. To this end, let us first prove the following.   

\begin{theorem} \label{thm: coloring core}
Let $G$ be a graph with maximum degree $\Delta$. If $G[\core(G)]$ has a proper \san{$\Delta$-edge} coloring, then $G$ has a matching that saturates $\core(G)$.
\end{theorem}

\begin{proof}
Let $n = |V(G)|$. Let $A = G[\core(G)]$ and $B = V(G) - V(A)$. Suppose that $G$ has no matching that saturates $A$. By Lemma \ref{lem: tutte lemma}, there exists $S \subseteq V(G)$ such that
\[
	\san{o_A(G-S) \geq |S| + 1\,.}
\]

    Now, we claim that if $C \subseteq A$ is an odd component of $G - S$, then
    \begin{equation} \label{eq: second inequality}
        e_G(C , S) \geq \Delta.
    \end{equation}
    Consider a proper $\Delta$-edge coloring of $A$ and extend it to a \san{$\Delta$-edge} coloring of $G - E(B)$ such that for every $u\in A$, all edges \san{incident} to $u$ have $\Delta$ distinct colors.

    If there are at most $\Delta - 1$ edges with one endpoint in $V(C)$ and another in $S$, then there exists a color $t$ which has not appeared on these edges. Since every vertex in $C$ has degree $\Delta$, the color $t$ appears at each vertex of $C$. Therefore, the edges with color $t$ in $C$ form a perfect matching in $C$, a contradiction, and the claim is proved.

    By the claim, we have 
    \[
        e_G(A - S, S) \geq \lp |S| + 1 \rp \Delta.
    \]
    Now, by the Pigeonhole Principle, there exists $u \in S$ such that 
    \[
        e_G(A, \la u \ra) \geq \Delta + \san{1\,,}
    \]
    which is a contradiction.
\end{proof}

A matching that saturates every vertex of $\core(G)$ is clearly a mitigating set, hence from Theorem~\ref{thm: coloring core} we immediately get: 

\begin{corollary}
    If $G$ is a graph of order $n$ such that $G[\core(G)]$ is Class 1, then $\es(G) \leq \frac{n}{2}$. In particular, if $G$ is Class 1 graph, then $\es(G) \leq \frac{n}{2}$.
\end{corollary}

In our final result, we prove that also in graphs with large maximum \san{degree,} the $\Delta$-edge stability number is bounded from the above by \san{half} of the order. To prove the \san{result,} we state the following \san{lemma,} for which we recall that a graph is \san{said to be} {\em factor critical} if every \san{vertex-deleted} subgraph has a perfect matching. 


\begin{lemma} \label{lem: factor critical}
    If $G$ does not have a perfect matching and $S\subseteq V(G)$ is the \san{largest} set such that 
    \begin{equation} \label{ineq: tutte ineq for factor critical lemma}
                o(G - S) > |S|,
    \end{equation}
    then each component of $G - S$ is factor critical.
\end{lemma}
\begin{proof}
     Suppose that $C$ is a component of $G - S$ of even order and $v \in C$. Note \san{that} $C - \{v\}$ has at least one odd component. Therefore, $S \cup \{v\}$ satisfies (\ref{ineq: tutte ineq for factor critical lemma}), a contradiction. So, $G - S$ has no even component. Suppose $C$ is an odd component of $G - S$ which is not factor critical. \san{Thus, for some $v \in V(C)$, $C - \{v\}$ does} not have a perfect matching. Hence, there is  $S' \subseteq \san{V(C)} - \{v\}$ such that $\san{o\left( (C -\{v\}) - S' \right)} > |S'| + 1$. Thus, $S \cup S' \cup \{v\}$ satisfies (\ref{ineq: tutte ineq for factor critical lemma}), a contradiction.
\end{proof}

Our final result now reads as follows.

\begin{theorem} \label{thm: dense graph}
Let $G$ be a connected graph of order $n$ with maximum degree $\Delta$. If $\Delta \geq\frac{n-2}{3}$, then $es_{\Delta}(G)\leq\lceil\frac{n}{2}\rceil$.
\end{theorem}
\begin{proof}
Let \san{$A=V(\core(G))$} and $B=V(G)\backslash A$. If there exists a matching in $G$ which saturates $A$, then we get the result. Now, assume $G$ does not contain such matchings. So, by Lemma \ref{lem: tutte lemma}, there exists $S\subseteq V(G)$ such that $\oa(G - S)\geq|S| + 1$. Thus, there exist at least $|S|+1$ odd components in $G -S$ that are contained \san{in $G[A]$.} We name them as $C_{1}, \ldots, C_{|S|+1}$.  \san{Let $\mathcal{C} = \{C_1, \dots, C_{|S|+1}\}$.}

We claim that if \san{$C \in \mathcal{C}$} and $e\left(V(C), S\right) < \Delta$, then \san{$|V(C)| \geq \Delta + 1$, and} moreover, there exists $v \in V(C)$ such that $N[v] \subseteq V(C)$.
First, suppose that \san{$|V(C)| \geq \Delta$}.  If every vertex of $C$ has a neighbor in $S$, then it is clear that $e(V(C),S)\geq\Delta$, a contradiction. So, there exists a vertex $v\in C$ such that $N[v]\subseteq V(C)$. Thus, $|V(C)|\geq\Delta+1$, as desired. \san{Suppose $|V(C)| < \Delta$. Since} every vertex of $C$ \san{has} degree $\Delta$, $e(\{u\},S)\geq\Delta-|V(C)|+1$ \san{for each $u \in V(C)$.} Hence $e(V(C),S)\geq |V(C)|(\Delta-|V(C)|+1)\geq\Delta$, \san{a} contradiction. So the claim is proved.

\san{Let $L = \bigcup_{i = 1}^{|S|+1} V(C_i)$.}
Note that if $e(V(C),S)\geq\Delta$ for every $C\in \san{\mathcal{C}}$, then $e(\san{L}, S)\geq\Delta(|S|+1)$ and by the pigeonhole \san{principle}, there exists $w\in S$ such that $d(w)\geq\Delta+1$, a contradiction.

Now, since $\Delta\geq\frac{n-2}{3}$, if \san{$C \in \mathcal{C}$} and $e(V(C),S) < \Delta$, then by the \san{claim,} we have $|V(C)|\geq\frac{n+1}{3}$. Hence, we have at most two $C_{i}$ with  $e(V(C_i), S) < \Delta$, say $C_1$ and $C_2$. \san{Consider} two cases. 

\medskip\noindent
{\bf Case 1}. $|V(C_{1})|>\Delta$ and for every $i>1$, $|V(C_{i})|\leq\Delta$. \\
So, by the claim, $e(L - V(C_1), S)\geq\Delta|S|$. On the other \san{hand, since $d(s)\leq\Delta$ for every $s\in S$,} we have $e(L - V(C_1), S)\leq\Delta|S|$. Hence,  $e(L - V(C_1), S)=\Delta|S|$\san{, meaning} that all neighbors of each vertex of $S$ \san{are in $L - V(C_1)$.} So, $C_{1}$ is a component of \san{$G$, but this contradicts the assumption that $G$ is connected.}

\medskip\noindent
{\bf Case 2}. $|V(C_{1})| >\Delta$, $|V(C_{2})|>\Delta$, and for every $i>2$, $|V(C_{i})| \leq \Delta$. \\
Since, $|V(C_{1})|\geq\frac{n+1}{3}$ and $|V(C_{2})|\geq\frac{n+1}{3}$, so $| V(G) - \left(V(C_1) \cup V(C_2)\right)|<\frac{n-1}{3}$. If there exists a vertex $u \in \san{G - S}$ of degree $\Delta$, \san{then} since $N[u]\subseteq V(G) - \left(V(C_1) \cup V(C_2)\right)$ and $|N[u]|>\frac{n}{3}$, we get a contradiction. So, $L -\left(V(C_1) \cup V(C_2)\right) = \varnothing$ and thus, $|S|=1$. \san{ Let $s$ be the member of $S$.}

Note that if \san{$S$ is chosen to be of} maximum size, \san{then} by Lemma~\ref{lem: factor critical}\san{, all members of $\mathcal{C}$}  are factor critical. We choose two arbitrary edges $su$ and $sv$ such that $u\in V(C_{1})$ and $v\in V(C_{2})$. As $C_{1}$ and $C_{2}$ are factor critical, there exist matchings $M$ and $N$ that saturate vertices of $C_{1}\backslash\{u\}$ and $C_{2}\backslash\{v\}$, respectively. Since all vertices of degree $\Delta$ are in $L \cup S$, it is clear that $M \cup N \cup \{su, sv \}$ is a mitigating set for $G$. Note that $|M| + |N| + |\{ su, sv \}| = \frac{|L| + 2}{2} \leq \frac{n+1}{2}$, thus $\es(G) \leq \frac{n+1}{2}$ and we are done.
\end{proof}

\begin{remark}
\normalfont
To see that the bound  $\Delta \geq\frac{n-2}{3}$ of Theorem~\ref{thm: dense graph} cannot be lowered in general, consider the following example. For each odd $t\ge 7$, let $G_t$ be the graph constructed as follows. Take three \san{vertex-disjoint} copies of $K_{t}$, remove one edge from each of them, add a vertex $u$, and join $u$ to the endpoints of the deleted \san{edges;} see Fig.~\ref{fig: tight example}. 

\begin{figure}[ht!] 
    \centering

\tikzset{every picture/.style={line width=0.75pt}} 

\begin{tikzpicture}[x=0.75pt,y=0.75pt,yscale=-1,xscale=1]

\draw   (60,115) .. controls (60,101.19) and (71.19,90) .. (85,90) .. controls (98.81,90) and (110,101.19) .. (110,115) .. controls (110,128.81) and (98.81,140) .. (85,140) .. controls (71.19,140) and (60,128.81) .. (60,115) -- cycle ;
\draw  [line width=3.75] [line join = round][line cap = round] (86,90.6) .. controls (86,90.6) and (86,90.6) .. (86,90.6) ;
\draw  [line width=3.75] [line join = round][line cap = round] (86,139.6) .. controls (86,139.6) and (86,139.6) .. (86,139.6) ;
\draw  [dash pattern={on 4.5pt off 4.5pt}]  (85,90) -- (85,140) ;
\draw   (128,115) .. controls (128,101.19) and (139.19,90) .. (153,90) .. controls (166.81,90) and (178,101.19) .. (178,115) .. controls (178,128.81) and (166.81,140) .. (153,140) .. controls (139.19,140) and (128,128.81) .. (128,115) -- cycle ;
\draw  [line width=3.75] [line join = round][line cap = round] (154,90.6) .. controls (154,90.6) and (154,90.6) .. (154,90.6) ;
\draw  [line width=3.75] [line join = round][line cap = round] (154,139.6) .. controls (154,139.6) and (154,139.6) .. (154,139.6) ;
\draw  [dash pattern={on 4.5pt off 4.5pt}]  (153,90) -- (153,140) ;
\draw   (197,113) .. controls (197,99.19) and (208.19,88) .. (222,88) .. controls (235.81,88) and (247,99.19) .. (247,113) .. controls (247,126.81) and (235.81,138) .. (222,138) .. controls (208.19,138) and (197,126.81) .. (197,113) -- cycle ;
\draw  [line width=3.75] [line join = round][line cap = round] (223,88.6) .. controls (223,88.6) and (223,88.6) .. (223,88.6) ;
\draw  [line width=3.75] [line join = round][line cap = round] (223,137.6) .. controls (223,137.6) and (223,137.6) .. (223,137.6) ;
\draw  [dash pattern={on 4.5pt off 4.5pt}]  (222,88) -- (222,138) ;
\draw  [line width=3.75] [line join = round][line cap = round] (290,111.6) .. controls (290,111.6) and (290,111.6) .. (290,111.6) ;
\draw    (222,138) .. controls (256,137.6) and (271,135) .. (289,111.6) ;
\draw    (222,88) .. controls (263,86.6) and (271,89.6) .. (289,111.6) ;
\draw    (153,90) .. controls (190,58.6) and (267,55.6) .. (289,111.6) ;
\draw    (153,140) .. controls (193,161.6) and (275,176) .. (289,111.6) ;
\draw    (85,140) .. controls (132,186) and (282,199) .. (289,111.6) ;
\draw    (85,90) .. controls (122,47.6) and (273,22.6) .. (289,111.6) ;

\draw (72,104.4) node [anchor=north west][inner sep=0.75pt]    {$e$};
\draw (53,74.4) node [anchor=north west][inner sep=0.75pt]    {$K_{t}$};
\draw (140,104.4) node [anchor=north west][inner sep=0.75pt]    {$e$};
\draw (121,74.4) node [anchor=north west][inner sep=0.75pt]    {$K_{t}$};
\draw (209,102.4) node [anchor=north west][inner sep=0.75pt]    {$e$};
\draw (190,72.4) node [anchor=north west][inner sep=0.75pt]    {$K_{t}$};
\draw (298,100.4) node [anchor=north west][inner sep=0.75pt]    {$u$};

\end{tikzpicture}

\caption{The graph $G_t$}
\label{fig: tight example}
\end{figure}
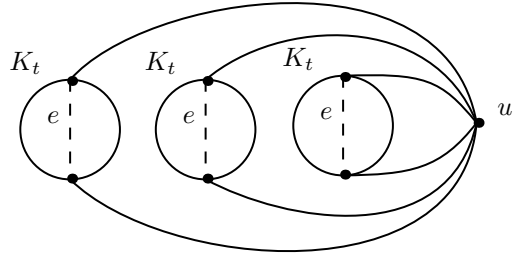

In $G_t$ we need at least $\frac{t+1}{2}$ edges to saturate the vertices of each $K_t$. Hence, $\es(G)\geq\frac{3(t+1)}{2}$. Setting $n = |V(G_t)|$, we have $n = 3t+1$. As $t\ge 7$ we have $\Delta(G_t)   = \frac{n-4}{3}$. In summary, 
$$\Delta(G_t) = \frac{n-4}{3} < \frac{n-2}{3}\quad {\rm and}\quad \es(G_t) \geq \frac{n}{2}+1\,,$$ 
hence the assumption on the maximum degree in Theorem~\ref{thm: dense graph} is tight. 
\end{remark}

As in a graph $G$ we have $\Delta(G) \ge \frac{2|E(G)|}{|V(G)|}$, Theorem~\ref{thm: dense graph} yields: 

\begin{corollary}
If $G$ is a graph of order $n$ and size at least $\frac{n(n-2)}{6}$, then $es_{\Delta}(G)\leq\lceil\frac{n}{2}\rceil$.
\end{corollary}

\section*{Acknowledgements}

We would like to thank one of the reviewers for a very careful reading of the paper and a lot of specific, useful advice on how to improve the presentation.

This research was in part supported by a grant from IPM (No.\ 1402050012). Sandi Klav\v zar was supported by the Slovenian Research and Innovation Agency ARIS (research core funding P1-0297 and projects N1-0285, N1-0355). 

\section*{Declaration of interests}
 
The authors declare that they have no conflict of interest. 

\section*{Data availability}
 
Our manuscript has no associated data.


\begin{thebibliography}{99}

\bibitem{akbari-2022}
S.~Akbari, A.~Beikmohammadi, S.~Klav\v{z}ar, N.~Movarraei, 
On the chromatic vertex stability number of graphs,
European J.\ Combin.\ 102 (2022) Paper 103504.

\bibitem{akbari-2023}
S.~Akbari, A.~Beikmohammadi, B.~Bre\v{s}ar, T.~Dravec, M.M.~Habibollahi, N.~Movarraei, 
On the chromatic edge stability index of graphs,
European J.\ Combin.\ 111 (2023) Paper 103690. 

\bibitem{akbari-2020}
S.~Akbari, S.~Klav\v{z}ar, N.~Movarraei, M.~Nahvi, 
Nordhaus-Gaddum and other bounds for the chromatic edge-stability number, 
European J.\ Combin.\ 84 (2020) Paper 103042.

\bibitem{alikhani-2023}
S.~Alikhani, M.R.~Piri, 
On the edge chromatic vertex stability number of graphs,
AKCE Int.\ J.\ Graphs Comb.\ 20 (2023) 29--34.

\bibitem{borg-2022} 
P.~Borg, 
Decreasing the maximum degree of a graph, 
Discrete Math.\ 345 (2022) Paper 112994.

\bibitem{borg-2017} 
P.~Borg, K.~Fenech, 
Reducing the maximum degree of a graph by deleting vertices,
Australas.\ J.\ Combin.\ 69 (2017) 29--40.

\bibitem{borg-2019} 
P.~Borg, K.~Fenech, 
Reducing the maximum degree of a graph by deleting edges,
Australas.\ J.\ Combin.\ 73 (2019) 247--260.

\bibitem{bresar-2020}
B.~Bre\v sar, S.~Klav\v zar, N. Movarraei, 
Critical graphs for the chromatic edge-stability number, 
Discrete Math.\ 343 (2020) Paper 111845. 

\bibitem{hajnal-1965}
A.~Hajnal,
A theorem on {$k$}-saturated graphs,
Canadian J.\ Math.\ 17 (1965) 720--724. 

\bibitem{henning-2007}
M.A.~Henning, A.~Yeo, 
Tight lower bounds on the size of a maximum matching in a regular graph, 
Graphs Combin.\ 23 (2007) 647--657.

\bibitem{jackson-1995}
B.~Jackson,
Cycles through vertices of large maximum degree,
J. Graph Theory 19 (1995) 157--168.

\bibitem{kano-1986}
M.~Kano, 
Factors of regular graphs,
J.\ Combin.\ Theory Ser.\ B 41 (1986) 27--36. 

\bibitem{kemnitz-2022}
A.~Kemnitz, M.~Marangio, 
On the {$\rho$}-edge stability number of graphs,
Discuss.\ Math.\ Graph Theory 42 (2022) 249--262.

\bibitem{kemnitz-2024}
A.~Kemnitz, M.~Marangio, 
On the vertex stability numbers of graphs, 
Discrete Appl.\ Math.\ 344 (2024) 1--9.

\bibitem{kemnitz-2018}
A.~Kemnitz, M.~Marangio, N.~Movarraei, 
On the chromatic edge stability number of graphs,
Graphs Combin.\ 34 (2018) 1539--1551.

\bibitem{knor-2022}
M.~Knor, M.~Petru\v{s}evski, R.~\v{S}krekovski, 
On chromatic vertex stability of $3$-chromatic graphs with maximum degree $4$,
Discrete Math.\ Lett.\ 10 (2022) 75--80.

\bibitem{lei-2023}
H.~Lei, X.~Lian, X.~Meng, Y.~Shi, Y.~Wang, 
On critical graphs for the chromatic edge-stability number,
Discrete Math.\ 346 (2023) Paper 113307.


\san{
\bibitem{tutte-1947}
W.T.~Tutte, 
The factorization of linear graphs.
J.\ London Math.\ Soc.\ 22 (1947) 107--111.
}

\bibitem{tutte-1952}
W.T.~Tutte, 
The factors of graphs,
Canad.\ J.\ Math.\ 4 (1952) 314--328. 

\bibitem{vizing-1964} 
V.G.~Vizing,
On an estimate of the chromatic class of a $p$-graph,
Diskret.\ Analiz 3 (1964) 25--30.

\end{thebibliography}
\end{document}